\documentclass{amsart}
\usepackage{amsmath,amsfonts,amsbsy,amsgen,amscd,mathrsfs,amssymb,amsthm}
\usepackage{mathtools}
\usepackage{hyperref}
\usepackage{enumerate}

\usepackage{pgfplots}

\numberwithin{equation}{section}
\newtheorem{theorem}{Theorem}[section]

\newtheorem{lemma}[theorem]{Lemma}

\theoremstyle{definition}
\newtheorem{assumption}[theorem]{Assumption}

\newtheorem{remark}[theorem]{Remark}
\newtheorem{notation}[theorem]{Notation}

\makeatletter\renewenvironment{proof}[1][\proofname] {\par\pushQED{\qed}\normalfont\topsep6\p@\@plus6\p@\relax\trivlist\item[\hskip\labelsep\bfseries#1\@addpunct{.}]\ignorespaces}{\popQED\endtrivlist}

\newcommand\1{\mathbf 1}
\newcommand\al{\alpha}
\newcommand\be{\beta}

\newcommand\dd{\mathrm{d}}
\newcommand\De{\Delta}

\newcommand\Ex{\mathbf E}
\newcommand\ex{\mathrm e}
\newcommand\eps{\varepsilon}
\newcommand\erf{\mathrm{erf}}
\newcommand\erfc{\mathrm{erfc}}
\newcommand\ga{\gamma}

\newcommand\ka{\kappa}
\newcommand\La{\Lambda}
\newcommand\NN{\mathbb N}
\newcommand\Om{\Omega}
\newcommand\om{\omega}
\newcommand\RR{\mathbb{R}}

\newcommand\ZZ{\mathbb{Z}}
\renewcommand\d{~\mathrm{d}}
\renewcommand\phi{\varphi}
\renewcommand\Pr{\mathbf{P}}

\begin{document}

\title[Joint and Conditional Local Limits with Occupation Measures]{Joint and Conditional Local Limit Theorems for Lattice Random Walks and Their Occupation Measures}
\author{Pierre Yves Gaudreau Lamarre}
\address{Princeton University, Princeton, NJ\space\space08544, USA}
\email{plamarre@princeton.edu}
\thanks{This work was partially supported by an NSERC Doctoral fellowship and a Gordon Y.~S.~Wu fellowship.}
\maketitle

\begin{abstract}
Let $S_n$ be a lattice random walk with mean zero and finite variance, and let $\La^a_n$
be its occupation measure at level $a$. In this note, we prove local limit theorems for
$\Pr[S_n=x,\La^a_n=\ell]$ and $\Pr[S_n=x|\La^a_n=\ell]$ in the cases where
$a$, $|x-a|$ and $\ell$ are either zero or at least of order $\sqrt n$.
The asymptotic description of these quantities matches the corresponding probabilities for Brownian
motion and its local time process.

This note can be seen as a generalization of previous results by Kaigh (1975) and Uchiyama (2011).
In similar fashion to these results, our method of proof relies on path decompositions
that reduce the problem at hand to the study of random walks with independent increments.
\end{abstract}

\section{Introduction}

Let $X$ be a random variable supported on the lattice $c+\ZZ$ ($|c|<1$)
with mean zero and variance $\nu\in(0,\infty)$. 
Let $X_1,X_2,X_3,\ldots$ be i.i.d. copies of $X$, and for each $n\in\NN$, let $S_n:=X_1+\cdots+X_n$.
The central limit theorem states that $S_n/\sqrt{n}$ converges in distribution as $n\to\infty$ to a
Gaussian law. A natural refinement of this result consists of the local limit theorem,
which describes the asymptotic behavior of $S_n$ on a finer scale (c.f., \cite[\S49]{GnKo}).

\begin{theorem}[Gnedenko's Local Limit Theorem]\label{Theorem: Local CLT}
Suppose that the lattice $c+\ZZ$ is maximal for $X$. It holds that
\[\sqrt{n}\Pr[S_n=nc+x]=\frac{\ex^{-(nc+x)^2/2n\nu}}{\sqrt{2\pi\nu}}+o(1)\]
as $n\to\infty$, where the error term $o(1)$ is uniform in $x\in\ZZ$.
\end{theorem}

Starting from the mid 1970s, there has been a growing interest in local limit theorems involving random
walks that are conditioned on some event of interest, or for other closely related processes.
Notable examples include local limit theorems for random walks conditioned to stay positive
\cite{BJD,Ca,VaWa}, as well as lattice random walks conditioned to avoid returning to the origin
\cite{KaT,Ka,U3} or visiting a particular point or finite set \cite{U1,U2}. Our aim in this note is to
contribute to this line of results by proving local limit theorems for lattice random walks considered
jointly with their occupation measures
\begin{align}\label{Equation: Occupation Measure}
\La_n^a:=\sum_{k=1}^n\1_{\{S_k=a\}},\qquad a\in \RR,~n\in\NN.
\end{align}

It is well known that, with the appropriate normalization, the couple $(S_n,\La_n^{\cdot})$ converges in joint
distribution to $(B^\nu_1,L^\cdot)$, where $(B^\nu_t)_{t\geq0}$ is a Brownian motion with variance $\nu$ and
$(L^a)_{a\in\RR}$ is its local time process on the time interval $[0,1]$ (see, for instance, \cite{BK}). It is
thus natural to expect that the large $n$ limits of the probabilities
\[\Pr[S_n=x,\La^a_n=\ell],\qquad x,a\in\RR,~\ell\in\NN\]
should be described by the joint distribution of $(B^\nu_1,L^a)_{a\in\RR}$. As we will see in Section
\ref{Section: Main Result}, where we state our main results, this is indeed the case.

As one might expect, conditioning a random walk to avoid or visit a fixed level before some time $n$
breaks the independence of increments. Thus, the methods used in the classical local limit
theorems, such as Theorem \ref{Theorem: Local CLT}, cannot be applied directly to joint or conditioned
local limit theorems. A typical workaround for this problem (used, for example, in
\cite{Ca,Ka,U1,VaWa}) is to use path transformations or decompositions to
translate the problem at hand into one about random walks with independent increments. Our methods
of proof rely on such techniques (see Lemma \ref{Lemma: Path Transformation}).

\subsection*{Acknowledgments}

The author thanks Mykhaylo Shkolnikov for his encouragement, guidance, and helpful
discussions.
%as well as Jessyka Gunville for editing this note for grammar and style.
The author thanks an anonymous referee for pointing out the work of Uchiyama \cite{U1,U2},
which supersedes a result proved in an earlier version of this note.

\section{Main Result}\label{Section: Main Result}

Throughout this note, we make the following assumption.

\begin{assumption}\label{Assumption: Centered Lattice}
$(X_k)_{k\in\NN}$ are i.i.d. copies of $X$, which is supported on $\ZZ$ and such that $\Ex[X]=0$ and
$\Ex[X^2]=\nu\in(0,\infty)$. $\ZZ$ is maximal for $X$, that is, there is no $c,h$ with $h>1$ such that $X$ is
supported on $c+h\ZZ$.
The random walk $S_n=X_1+\cdots+X_n$ ($n\in\NN$) is recurrent (i.e., $S_n=0$ occurs for infinitely many $n$ almost surely), and aperiodic
(i.e., for every $x\in\ZZ$, there exists $n\in\NN$ such that $\Pr[S_m=x]>0$ for every $m\geq n$).
\end{assumption}

We also use the following notation.

\begin{notation}
We use $C$ to denote constants independent of $n$
whose values may change from line to line. We denote the error functions
\[\erf(x):=\frac2{\sqrt{\pi}}\int_0^x\ex^{-y^2}\d y
\quad\text{and}\quad
\erfc(x):=1-\erf(x),\qquad x\in\RR,\]
as well as the functions
\begin{align*}
\phi_\nu(x,a,\ell)&:=\frac{(|a|+|x-a|+\nu\ell)\ex^{-(|a|+|x-a|+\nu\ell)^2/2\nu}}{\sqrt{2\pi\nu}},& x,a\in\RR,~\ell>0,\\
\psi_\nu(a,\ell)&:=2\nu\frac{\ex^{-(|a|+\nu\ell)^2/2\nu}}{\sqrt{2\pi\nu}},& a\in\RR,~\ell>0.
\end{align*}
Given two sequences $(\al_n)_{n\in\NN}$ and $(\be_n)_{n\in\NN}$, we use
$\al_n\sim\be_n$ as $n\to\infty$ to denote that $\lim_{n\to\infty}\al_n/\be_n=1$.
\end{notation}

Our result is the following.

\begin{theorem}\label{Theorem: Main}
Under Assumption \ref{Assumption: Centered Lattice}, for every $\ka>0$, it holds that
\begin{align}
n\Pr[S_n=x,\La^a_n=\ell]
&=\phi_\nu\left(\frac{a}{\sqrt{n}},\frac{x}{\sqrt{n}},\frac{\ell}{\sqrt{n}}\right)
+o(1)\quad\text{and}
\label{Equation: Main Theorem 1}\\
\sqrt{n}\Pr[\La^a_n=\ell]
&=\psi_\nu\left(\frac{a}{\sqrt{n}},\frac{\ell}{\sqrt{n}}\right)+o(1)
\label{Equation: Main Theorem 2}
\end{align}
as $n\to\infty$, where the error terms $o(1)$ are uniform in $x,a\in\ZZ$ and $\ell\in\NN$ such that
\begin{enumerate}
\item $\ell\geq\ka\sqrt{n}$;
\item $a=0$ or $|a|\geq\ka\sqrt{n}$; and
\item $|x-a|=0$ or $|x-a|\geq\ka\sqrt{n}$.
\end{enumerate}
\end{theorem}

It can be noted that \eqref{Equation: Main Theorem 1} is consistent with the joint density
\begin{align*}
\Pr[B^\nu_1\in\dd x,L^a\in\dd \ell]
=\phi_\nu(a,x,\ell)\d x\dd\ell,\qquad x,a\in\RR,~\ell>0
\end{align*}
(see \cite[Page 6]{Bo2} and \cite{Pi}).
As for \eqref{Equation: Main Theorem 2}, by integrating the above with respect to $x$, one easily infers that
\[\Pr[L^a\in\dd \ell]=\psi_\nu(a,\ell)\d\ell,\qquad a\in\RR,~\ell>0.\]

Any statement of the form \eqref{Equation: Main Theorem 1} appears to be new. As for
\eqref{Equation: Main Theorem 2}, our statement can essentially be obtained from \cite[Theorem 1.1]{Bo1},
though there are notable differences in the assumptions we make in this note; for example, in \cite{Bo1}, $X$
is assumed to have finite moments of all orders, whereas we assume only a finite second moment.
Our motivation for including \eqref{Equation: Main Theorem 2} is that, by combining it with
\eqref{Equation: Main Theorem 1}, we obtain a local limit theorem for $S_n$ conditioned on its local time process.

\section{Similar Results}

There are two results in the literature which can be thought of as the closest analogues of Theorem \ref{Theorem: Main}.
The results in question concerns local limit theorems for the probabilities of the form $\Pr[S_n=x,\La^a_n=\ell]$ in the case where $\ell=0$.
In order to state these results, we need the following notation.

\begin{notation}
Let us denote the first hitting times
\[T_a:=\inf\{t>0:B^\nu_t=a\}
\quad\text{and}\quad
\tau_a:=\min\{n\geq1:S_n=a\},\qquad a\in\RR\]
so that $\{L^a=0\}=\{T_a>1\}$ and $\{\La^a_n=0\}=\{\tau_a>n\}$.
We denote the Gaussian density by
\[\ga_\nu(x):=\frac{\ex^{-x^2/2\nu}}{\sqrt{2\pi\nu}},\qquad x\in\RR.\]
\end{notation}

The first result is due to Kaigh \cite[Theorem 1, (9) and Section 4]{Ka}:

\begin{theorem}[Kaigh]\label{Theorem: Hitting Away Local}
Suppose that Assumption \ref{Assumption: Centered Lattice} holds.
For every $a\in\ZZ$ and $n\in\NN$, it holds that 
\[\Pr[S_n=a,\tau_0>n]=\Pr[\tau_a=n].\]
Moreover,
\begin{align}\label{Equation: Hitting Away Local}
\lim_{n\to\infty}\sup_{a\in\ZZ}\left|n\Pr[S_n=a,\tau_0>n]-\frac{|a|\ex^{-a^2/2\nu n}}{\sqrt{2\pi \nu n}}\right|&=0,\\
\label{Equation: Hitting Away Local Conditional}
\lim_{n\to\infty}\sup_{a\in\ZZ}\left|\sqrt{n}\Pr[S_n=a|\tau_0>n]-\frac{|a|\ex^{-a^2/2\nu n}}{2\nu\sqrt{n}}\right|&=0.
\end{align}
\end{theorem}

The function appearing in \eqref{Equation: Hitting Away Local Conditional} is the
two-sided Rayleigh density,
which is the endpoint distribution of a two-sided Brownian meander:
\[\Pr[B^\nu_1\in\dd x|T_0>1]=\frac{|x|\ex^{-x^2/2\nu}}{2\nu}\d x,\qquad x\in\RR.\]
We note that the main result of \cite{Ka} is \eqref{Equation: Hitting Away Local},
and that \eqref{Equation: Hitting Away Local Conditional} can be obtained from
the former when combined with the well-known asymptotic (c.f., \cite[(1)]{Ka})
\[\lim_{n\to\infty}\sqrt{n}\Pr[\tau_0>n]=\sqrt{2\nu/\pi}.\]

The second result we mention is a special case of a more general statement due to Uchiyama \cite[Theorem 1.1]{U1}:

\begin{theorem}[Uchiyama]\label{Theorem: Uchiyama}
Suppose that Assumption \ref{Assumption: Centered Lattice} holds.
For every $\ka>0$, it holds that
\begin{align}\label{Equation: Avoiding Away 1}
\sqrt{n}\Pr[S_n=x,\La^a_n=0]
=\ga_\nu\left(\frac{x}{\sqrt{n}}\right)-\ga_\nu\left(\frac{2a-x}{\sqrt{n}}\right)+o(1)
\end{align}
as $n\to\infty$, where the error term $o(1)$ is uniform in $x,a\in\ZZ$ such that
\begin{enumerate}
\item $a\geq\ka \sqrt{n}$ and $x\leq a$; or
\item $a\leq-\ka \sqrt{n}$ and $x\geq a$.
\end{enumerate}
\end{theorem}

Equation \eqref{Equation: Avoiding Away 1} is consistent with
\[\Pr[B^\nu_1\in\dd x,T_a>1]=\ga_\nu(x)-\ga_\nu(2a-x)\d x,\]
which holds for $a>0$ and $x\leq a$, or $a<0$ and $x\geq a$.
(This is easily computed by using the joint distribution of $B^\nu_1$ and $\sup\{B^\nu_t:t\in[0,1]\}$.)

\begin{remark}
Theorems \ref{Theorem: Hitting Away Local} and \ref{Theorem: Uchiyama} both rely on path transformations or representations
that relate the probabilities under consideration to probabilities involving an unconditioned/unconstrained
random walk (c.f., \cite[(15)]{Ka} and \cite[(3.2)]{U1}).
\end{remark}

\section{Proof Strategy}\label{Section: Strategy}

We begin with the path decomposition that serves as the basis of our proof. The idea, illustrated in
Figure \ref{Figure: Path Decomposition} below, is to break down a random walk path that visits some level $a$
a total of $\ell$ times before time $n$ into three parts.
\begin{figure}[htbp]
\begin{center}
\begin{tikzpicture}

\begin{axis}[%
width=4in,
height=2in,
at={(1.011in,0.642in)},
scale only axis,
xmin=0,
xmax=100,
ymin=-4,
ymax=8,
axis background/.style={fill=white}
]
\addplot [color=red, dashed, forget plot]
  table[row sep=crcr]{%
0	4\\
100	4\\
};
\addplot [color=blue, forget plot, thick]
  table[row sep=crcr]{%
0	0\\
1	1\\
2	2\\
3	1\\
4	2\\
5	1\\
6	0\\
7	-1\\
8	-2\\
9	-3\\
10	-4\\
11	-3\\
12	-2\\
13	-1\\
14	0\\
15	-1\\
16	-2\\
17	-1\\
18	-2\\
19	-1\\
20	0\\
21	-1\\
22	0\\
23	-1\\
24	0\\
25	1\\
26	0\\
27	-1\\
28	0\\
29	1\\
30	2\\
31	3\\
32	4\\
33	5\\
34	6\\
35	7\\
36	6\\
37	5\\
38	4\\
39	5\\
40	6\\
41	7\\
42	6\\
43	7\\
44	6\\
45	7\\
46	6\\
47	7\\
48	8\\
49	7\\
50	8\\
51	7\\
52	6\\
53	7\\
54	6\\
55	5\\
56	4\\
57	5\\
58	6\\
59	7\\
60	6\\
61	5\\
62	6\\
63	5\\
64	4\\
65	3\\
66	2\\
67	3\\
68	4\\
69	3\\
70	2\\
71	1\\
72	0\\
73	1\\
74	0\\
75	1\\
76	0\\
77	-1\\
78	0\\
79	1\\
80	0\\
81	1\\
82	2\\
83	3\\
84	2\\
85	1\\
86	0\\
87	-1\\
88	-2\\
89	-3\\
90	-4\\
91	-3\\
92	-4\\
93	-3\\
94	-2\\
95	-3\\
96	-2\\
97	-1\\
98	0\\
99	-1\\
100	-2\\
};
\addplot [color=black, draw=none, mark=x, mark options={solid, black}, forget plot,thick]
  table[row sep=crcr]{%
32	4\\
};
\addplot [color=black, draw=none, mark=x, mark options={solid, black}, forget plot,thick]
  table[row sep=crcr]{%
38	4\\
};
\addplot [color=black, draw=none, mark=x, mark options={solid, black}, forget plot,thick]
  table[row sep=crcr]{%
56	4\\
};
\addplot [color=black, draw=none, mark=x, mark options={solid, black}, forget plot,thick]
  table[row sep=crcr]{%
64	4\\
};
\addplot [color=black, draw=none, mark=x, mark options={solid, black}, forget plot,thick]
  table[row sep=crcr]{%
68	4\\
};
\end{axis}
\draw [thick,decorate,decoration={brace,amplitude=5pt}]
(2.58,5.1) -- (5.8,5.1) node [black,midway,yshift=13]
{Part 1};
%\draw [thick,decorate,decoration={brace,mirror,amplitude=4pt}]
%(5.8,4.9) -- (6.43,4.9) node [black,midway,yshift=-10]
%{$\tau_1$};
%\draw [thick,decorate,decoration={brace,mirror,amplitude=5pt}]
%(6.43,4.9) -- (8.25,4.9) node [black,midway,yshift=-10]
%{$\tau_2$};
%\draw [thick,decorate,decoration={brace,mirror,amplitude=5pt}]
%(8.25,4.9) -- (9.07,4.9) node [black,midway,yshift=-10]
%{$\tau_3$};
%\draw [thick,decorate,decoration={brace,amplitude=2pt}]
%(9.07,5.1) -- (9.47,5.1) node [black,midway,yshift=10]
%{$\tau_4$};
\draw [thick,decorate,decoration={brace,mirror,amplitude=10pt}]
(5.8,4.9) -- (9.47,4.9) node [black,midway,yshift=-18]
{Part 2};
\draw [thick,decorate,decoration={brace,amplitude=5pt}]
(9.47,5.1) -- (12.715,5.1) node [black,midway,yshift=13]
{Part 3};
\end{tikzpicture}%
\caption{Path Decomposition for $n=100$, $a=4$, and $\ell=5$.}
\label{Figure: Path Decomposition}
\end{center}
\end{figure}
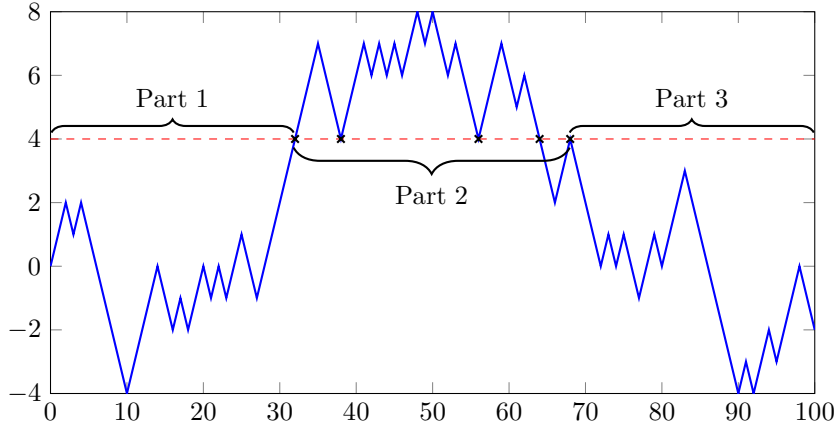
The first part consists of the steps before the first visit to level $a$
(this is vacuous if $a=0$).
If $x\neq a$, these steps are counted by the first hitting time $\tau_a$.
The second part consists of the $\ell-1$ excursions of the random walk
between consecutive visits to level $a$.
The number of such steps is equal in distribution to a sum of $\ell-1$ i.i.d. copies of $\tau_0$.
The third part consists of the steps after the last visit to level $a$ (this is vacuous if $x=a$).
If $x\neq a$, then this last part behaves like a random walk
that avoids returning to zero. More precisely, we have the following statement.

\begin{lemma}\label{Lemma: Path Transformation}
Let $(\om_k)_{k\in\NN}$ be i.i.d. copies of the hitting time $\tau_0$. For evety $u\in\NN$,
let $\Om_u:=\om_1+\cdots+\om_u$.
For every $a,x\in\ZZ$ and $\ell\in\NN$, we have that
\begin{multline*}
\Pr[S_n=x,\La^a_n=\ell]\\
=\begin{cases}
\displaystyle
\sum_{k=1}^n\Pr[\tau_a=k]\sum_{m=1}^{n-k}\Pr[\Om_{\ell-1}=m]\Pr[\tau_{x-a}=n-k-m]&\text{if }a\neq0,~a\neq x\\
\displaystyle
\sum_{m=1}^{n}\Pr[\Om_\ell=m]\Pr[\tau_x=n-m]&\text{if }a=0,~a\neq x\\
\displaystyle
\sum_{k=1}^n\Pr[\tau_a=k]\Pr[\Om_{\ell-1}=n-k]&\text{if }a\neq0,~a= x\\
\displaystyle
\sum_{k=1}^n\Pr[\Om_\ell=n]&\text{if }a=0,~a=x
\end{cases}
\end{multline*}
and
\[\Pr[\La^a_n=\ell]
=\begin{cases}
\displaystyle
\sum_{k=1}^n\Pr[\tau_a=k]\sum_{m=1}^{n-k}\Pr[\Om_{\ell-1}=m]\Pr[\tau_{0}>n-k-m]&\text{if }a\neq0\\
\displaystyle
\sum_{m=1}^{n}\Pr[\Om_\ell=m]\Pr[\tau_0>n-m]&\text{if }a=0.
\end{cases}\]
\end{lemma}
\begin{proof}
We focus on the case where $a\neq0$ and $a\neq x$, since the other cases can be proved
with a simpler version of the same argument. According to the strong Markov property,
\begin{multline*}
\Pr[S_n=x,\La^a_n=\ell]
=\sum_{k=1}^n\Pr[\tau_a=k,S_n=x,\La_n^a=\ell]\\
=\sum_{k=1}^n\Pr[\tau_a=k]\Pr[S_{n-k}=x-a,\La_{n-k}^0=\ell-1],
\end{multline*}
and similarly
\[\Pr[\La^a_n=\ell]=\sum_{k=1}^n\Pr[\tau_a=k]\Pr[\La_{n-k}^0=\ell-1].\]
By repeating the same procedure with $\ell-1$ successive return times to zero, we then obtain that
\[\Pr[S_{n-k}=x-a,\La_{n-k}^0=\ell-1]
=\sum_{m=1}^{n-k}\Pr[\Om_{\ell-1}=m]\Pr[S_{n-k-m}=x-a,\tau_0>n-k-m]\]
and
\[\Pr[\La_{n-k}^0=\ell-1]
=\sum_{m=1}^{n-k}\Pr[\Om_{\ell-1}=m]\Pr[\tau_0>n-k-m].\]
The result then follows from the fact that
\[\Pr[S_u=y,\tau_0>u]=\Pr[\tau_y=u],\qquad y\in\ZZ,~u\in\NN,\]
as stated in Theorem \ref{Theorem: Hitting Away Local}.
\end{proof}

The proof of Theorem \ref{Theorem: Main} is thus mostly reduced to a careful analysis of the distributions
of $\tau_a$ and $\Om_\ell$. We end this section by stating several results in this
direction that are needed in the proof of Theorem \ref{Theorem: Main}. The first such result concerns local
limits for $\Om_\ell$.

\begin{lemma}\label{Lemma: Omega Local Limit}
It holds that
\begin{align}
\lim_{n\to\infty}\sup_{\substack{m,\ell\in\NN\\\ell\geq\ka\sqrt{n}}}
\left|n\Pr[\Om_{\ell-1}=m]-\frac{(\nu\ell)\ex^{-(\nu\ell)^2/2\nu n(m/n)}}{\sqrt{2\pi\nu n (m/n)^3}}\right|&=0.
\label{Equation: Omega Local Limit}
\end{align}
\end{lemma}
\begin{proof}
According to \cite[Equation (10.6)]{Ke},
\begin{align}\label{Equation: Hitting Zero Deviation}
1-\Pr[\tau_0\leq n]\sim\frac{n^{-1/2}\sqrt{2\nu}}{\Gamma(1/2)}=\frac{\sqrt{2\nu}}{\sqrt{\pi n}}
\qquad\text{as }n\to\infty.
\end{align}
Therefore, it follows from \cite[Section XIII.6]{Fe}
that $\tau_0$ is in the domain of attraction of the
distribution with Laplace transform $\ex^{-\sqrt{t}}$ ($t\geq0$)
and with normalization sequence $2\nu n^2$ ($n\in\NN$). It is easy to see that the L\'evy distribution
satisfies this condition:
\[\int_0^\infty\frac{\ex^{-1/4y}}{\sqrt{4\pi y^3}}\cdot\ex^{-t y}\d y=\ex^{-\sqrt{t}},\qquad t\geq0,\]
and thus a straightforward change of variables implies that
\[\lim_{u\to\infty}\Pr\left[\frac{\Om_u}{u^2}\in\dd y\right]=\frac{\sqrt{\nu}\ex^{-\nu/2y}}{\sqrt{2\pi y^3}}\d y.\]

Clearly, $\tau_0$ is supported on the lattice $\ZZ$. In fact, $\ZZ$ is maximal for
$\tau_0$: If $\tau_0$ is supported on $m\ZZ$ for some $m>1$, then
$\Pr[S_n=0]>0$ only if $n$ is a multiple of $m$, which contradicts that $S_n$ is aperiodic.

By combining the above convergence in distribution with the fact that $\ZZ$ is maximal for $\tau_0$,
we conclude from the general local limit theorem for
stable distributions (c.f., \cite[Theorem in \S50]{GnKo}) that
\begin{align}\label{Equation: Omega Local Limit 0}
u^2\Pr[\Om_u=m]=\frac{u^3\sqrt{\nu}\ex^{-\nu u^2/2m}}{\sqrt{2\pi m^3}}+o(1)\qquad\text{as }u\to\infty,
\end{align}
with $o(1)$ uniform in $m\in\NN$. If we let $\ka_{\ell,n}:=(\ell-1)/\sqrt{n}\geq\ka+n^{-1/2}$, then
it follows from \eqref{Equation: Omega Local Limit 0} that
\begin{align*}
n\Pr[\Om_{\ell-1}=m]&=(\ka_{\ell,n}\sqrt{n})^2\Pr[\Om_{\ka_{\ell,n}\sqrt{n}}=m]/\ka_{\ell,n}^2\\
&=\frac{\ka_{\ell,n}n^{3/2}\sqrt{\nu}\ex^{-\nu (\ka_{\ell,n}\sqrt{n})^2/2m}}{\sqrt{2\pi m^3}}+o(1)\\
&=\frac{(\nu\ell) \ex^{-(\nu\ell)^2/2\nu n(m/n)}}{\sqrt{2\pi\nu n(m/n)^3}}+o(1),
\end{align*}
with $o(1)$ uniform in $m\in\NN$.
\end{proof}

We also need the following uniform Riemann sum approximation result.\footnote{We note that \cite[Lemma 2.21]{KaT}
is stated only for $f:[c_1,c_2]\times\RR^d\to\RR$ with $d=1$, and without the supremum over
$\al$. However, a trivial modification of their argument gives the result stated here.}

\begin{lemma}[{\cite[Lemma 2.21]{KaT}}]\label{Lemma: Uniform Riemann Approximation}
Let $f:[c_1,c_2]\times\RR^d\to\RR$ ($d\in\NN$) be a continuous function,
where $c_1<c_2\in\RR$. If
\[\lim_{\|y\|\to\infty}\sup_{c_1\leq u\leq c_2}|f(u,y)|=0,\]
then
\[\lim_{n\to\infty}\sup_{\substack{y\in\RR^d\\c_1\leq\al\leq c_2}}
\left|\frac1n\sum_{k=\lfloor nc_1\rfloor}^{\lfloor n\al\rfloor-1}
f(k/n,y)-\int_{c_1}^\al f(u,y)\d u\right|=0.\]
\end{lemma}

We finish Section \ref{Section: Strategy} with the following result on the tails of $\tau_a$,
which is a straightforward consequence of Theorem \ref{Theorem: Hitting Away Local} (see also \cite{U3}
and \cite[Corollary 2]{Ka}).

\begin{lemma}\label{Lemma: Tail Bounds}
For every $\ka>0$, it holds that
\begin{align}\label{Equation: Avoiding Away 2}
\Pr[\La^a_n=0]=\erf\left(\frac{|a|}{\sqrt{2\nu n}}\right)+o(1)
\end{align}
as $n\to\infty$ with $o(1)$ uniform in $a\in\ZZ$ such that $|a|\geq\ka\sqrt{n}$.
\end{lemma}
\begin{proof}
Given that
\[\int_0^1\frac{|a|\ex^{-a^2/2\nu nu}}{\sqrt{2\pi \nu nu^3}}\d u
=\erfc\left(\frac{|a|}{\sqrt{2\nu n}}\right),\qquad a\in\ZZ,\]
to show \eqref{Equation: Avoiding Away 2}, we need only prove that
\begin{align}\label{Equation: Hitting Time Away from Zero 2}
\lim_{n\to\infty}\sup_{|a|\geq\ka \sqrt{n}}\left|\sum_{k=1}^n\Pr[\tau_a=k]
-\int_0^1\frac{|a|\ex^{-a^2/2\nu nu}}{\sqrt{2\pi \nu n u^3}}\d u\right|=0.
\end{align}

For every $\De\in(0,1)$, \eqref{Equation: Hitting Time Away from Zero 2} is bounded by the
sum of the limits
\begin{align*}
A_1&:=\lim_{n\to\infty}\sup_{a\in\ZZ}	\left|\sum_{k=\lfloor n\De\rfloor+1}^n\Pr[\tau_a=k]
-\int_{\De}^1\frac{|a|\ex^{-a^2/2\nu nu}}{\sqrt{2\pi\nu n u^3}}\d u\right|,\\
A_2&:=\limsup_{n\to\infty}\sup_{|a|\geq\ka \sqrt{n}}
\int_0^\De\frac{|a|\ex^{-a^2/2\nu nu}}{\sqrt{2\pi\nu n u^3}}\d u,\text{ and}\\
A_3&:=\limsup_{n\to\infty}\sup_{|a|\geq\ka \sqrt{n}}\Pr[\tau_a\leq n\De].
\end{align*}
Our strategy is to prove that $A_1+A_2+A_3$ can be made arbitrarily small by taking $\De\to0$,
which then implies that \eqref{Equation: Hitting Time Away from Zero 2} holds.

For $A_1$, according to \eqref{Equation: Hitting Away Local},
\begin{align}\label{Equation: De must be fixed 1}
\sum_{k=\lfloor n\De\rfloor+1}^n\Pr[\tau_a=k]
=\sum_{k=\lfloor n\De\rfloor+1}^n\left(\frac{|a|\ex^{-a^2/2\nu k}}{\sqrt{2\pi\nu k^3}}+\frac{o(1)}k\right),
\end{align}
as $n\to\infty$, with $o(1)$ uniform in $a\in\ZZ$ and $\lfloor n\De\rfloor+1\leq k\leq n$. Thus,
\begin{align}\label{Equation: De must be fixed 2}
\sum_{k=\lfloor n\De\rfloor+1}^n\Pr[\tau_a=k]
=\frac1n\sum_{k=\lfloor n\De\rfloor+1}^n\frac{|a|\ex^{-a^2/2\nu n(k/n)}}{\sqrt{2\pi\nu n (k/n)^3}}+o(1),
\end{align}
with $o(1)$ uniform in $a\in\ZZ$.
If we apply Lemma \ref{Lemma: Uniform Riemann Approximation} with the function
\[f(u,y):=\frac{|y|\ex^{-y^2/2u}}{\sqrt{2\pi u^3}},\qquad (u,y)\in[\De,1]\times\RR,\]
then we get the desired result: For any $\De\in(0,1)$,
\[A_1=\lim_{n\to\infty}\sup_{a\in\ZZ}\left|\frac1n\sum_{k=\lfloor n\De\rfloor+1}^n
\frac{|a|\ex^{-a^2/2\nu n(k/n)}}{\sqrt{2\pi\nu n(k/n)^3}}
-\int_{\De}^1\frac{|a|\ex^{-a^2/2\nu nu}}{\sqrt{2\pi\nu n u^3}}\d u\right|=0.\]

For $A_2$, since $\erfc$ is strictly decreasing, a direct computation shows that
\[\sup_{|a|\geq\ka \sqrt{n}}\int_0^\De\frac{|a|\ex^{-a^2/2\nu nu}}{\sqrt{2\pi\nu n u^3}}\d u
=\sup_{|a|\geq\ka \sqrt{n}}\erfc\left(\frac{|a|}{\sqrt{2\nu n\De}}\right)=\erfc\left(\frac{\ka}{\sqrt{2\nu\De}}\right).\]
This vanishes as $\De\to0$.

Finally, to control $A_3$, we use the following result due to Kaigh \cite[Corollary 2]{Ka}:
For every fixed $z\geq0$, it holds that
\begin{align}\label{Equation: Kaigh Tail Bound}
\lim_{n\to\infty}\Pr[\tau_n\leq n^2z]=\erfc\left(\frac{1}{\sqrt{2\nu z}}\right).
\end{align}
For every $|a|\geq\ka\sqrt{n}$, let $\ka_a\geq\ka$ be such that $|a|=\ka_a\sqrt{n}$. Then,
\[\Pr[\tau_a\leq n\De]=\Pr[\tau_{\ka_a\sqrt{n}}\leq \ka_a^2n(\De/\ka_a^2)]
\leq\Pr[\tau_{\ka_a\sqrt{n}}\leq \ka_a^2n(\De/\ka^2)].\]
Thus, it follows from \eqref{Equation: Kaigh Tail Bound} that
\[A_3\leq\limsup_{n\to\infty}\Pr[\tau_{n}\leq n^2(\De/\ka^2)]=\erfc\left(\frac{\ka}{\sqrt{2\nu\De}}\right),\]
which vanishes as $\De\to0$.
\end{proof}

\begin{remark}\label{Remark: Hitting Time Away from Zero Rescaled}
It is easy to see that for any $\De>0$, the arguments used to prove \eqref{Equation: Avoiding Away 2}
also imply that
\[\lim_{n\to\infty}\sup_{|a|\geq\ka \sqrt{n}}\left|\Pr[\tau_a>\lfloor n\De\rfloor]
-\erf\left(\frac{|a|}{\sqrt{2\nu n\De}}\right)\right|=0.\]
\end{remark}

\section{Proof of Theorem \ref{Theorem: Main}}\label{Section: Proof}

We focus our proof on the case $x,|x-a|\geq\ka\sqrt{n}$, that is, we want to prove that as $n\to\infty$,
we have the limits
\begin{align}
\sup_{\substack{x,a\in\ZZ,~\ell\in\NN \\ |a|,|x-a|,\ell\geq\ka\sqrt{n}}}
\left|n\Pr[S_n=x,\La^a_n=\ell]
-\phi_\nu\left(\frac{a}{\sqrt{n}},\frac{x}{\sqrt{n}},\frac{\ell}{\sqrt{n}}\right)\right|&=o(1)\quad\text{and}
\label{Equation: Main Theorem 1 in Section}\\
\sup_{\substack{a\in\ZZ,~\ell\in\NN \\ |a|,\ell\geq\ka\sqrt{n}}}
\left|\sqrt{n}\Pr[\La^a_n=\ell]
-\psi_\nu\left(\frac{a}{\sqrt{n}},\frac{\ell}{\sqrt{n}}\right)\right|&=o(1).
\label{Equation: Main Theorem 2 in Section}
\end{align}
Indeed, the proofs for the cases where
$a=0$ or $|x-a|=0$ use simpler versions of the same argument, hence they are omitted.

\subsection{Proof of \eqref{Equation: Main Theorem 1 in Section}}

\subsubsection{Step 1}\label{Section: Step 1}

Our first step is to show that
for every fixed $\De\in(0,1)$, one has
\begin{multline}\label{Equation: Main 1 Proof Part 1}
\lim_{n\to\infty}
\sup_{\substack{x,a\in\ZZ,~k,\ell\in\NN \\ |x-a|,\ell\geq\ka\sqrt{n} \\ k\leq\lfloor(1-\De)n\rfloor}}
\Bigg|n\sum_{m=1}^{n-k}\Pr[\Om_{\ell-1}=m]\Pr[\tau_{x-a}=n-k-m]\\
-\frac{(|x-a|+\nu\ell)\ex^{-(|x-a|+\nu\ell)^2/2\nu n(1-k/n)}}{\sqrt{2\pi\nu n(1-k/n)^3}}\Bigg|=0.
\end{multline}
Define
\[I_{t,\tilde\De}(y,z):=\int_{\tilde\De}^{(1-t)(1-\tilde\De)}\frac{y\ex^{-y^2/2u}}{\sqrt{2\pi u^3}}
\frac{z\ex^{-z^2/2(1-t-u)}}{\sqrt{2\pi (1-t-u)^3}}\d u,\qquad y,z>0,~t,\tilde\De\in[0,1).\]
According to Lemma \ref{Lemma: Integral 1}, it holds that
\begin{align}\label{Equation: Main 1 Integral}
I_{t,0}(y,z)=\frac{(y+z)\ex^{-(y+z)^2/2(1-t)}}{\sqrt{2\pi(1-t)^3}},\qquad y,z>0,~t\in[0,1).
\end{align}
Thus, \eqref{Equation: Main 1 Proof Part 1} is bounded above by the sum of the terms
\begin{align*}
B_1&:=\sup_{n\in\NN}
\sup_{\substack{x,a\in\ZZ,~k,\ell\in\NN \\ |x-a|,\ell\geq\ka\sqrt{n} \\ k\leq\lfloor(1-\De)n\rfloor}}
\left|I_{k/n,0}\left(\frac{\nu\ell}{\sqrt{\nu n}},\frac{|x-a|}{\sqrt{\nu n}}\right)
-I_{k/n,\tilde\De}\left(\frac{\nu\ell}{\sqrt{\nu n}},\frac{|x-a|}{\sqrt{\nu n}}\right)\right|,\\
B_2&:=\limsup_{n\to\infty}
\sup_{\substack{x,a\in\ZZ,~k,\ell\in\NN \\ |x-a|,\ell\geq\ka\sqrt{n} \\ k\leq\lfloor(1-\De)n\rfloor}}
\left|n\sum_{m=\lfloor(n-k)(1-\tilde\De)\rfloor+1}^{n-k}\Pr[\Om_{\ell-1}=m]\Pr[\tau_{x-a}=n-k-m]\right|,\\
B_3&:=\limsup_{n\to\infty}
\sup_{\substack{x,a\in\ZZ,~k,\ell\in\NN \\ |x-a|,\ell\geq\ka\sqrt{n} \\ k\leq\lfloor(1-\De)n\rfloor}}
\left|n\sum_{m=1}^{\lfloor n\tilde\De\rfloor}\Pr[\Om_{\ell-1}=m]\Pr[\tau_{x-a}=n-k-m]\right|,\text{ and}
\end{align*}
\begin{multline*}
B_4:=\lim_{n\to\infty}
\sup_{\substack{x,a\in\ZZ,~k,\ell\in\NN \\ \ell\geq\ka\sqrt{n}}}
\Bigg|n\sum_{m=\lfloor n\tilde\De\rfloor+1}^{\lfloor(n-k)(1-\tilde\De)\rfloor}\Pr[\Om_{\ell-1}=m]\Pr[\tau_{x-a}=n-k-m]\\
-I_{k/n,\tilde\De}\left(\frac{\nu\ell}{\sqrt{\nu n}},\frac{|x-a|}{\sqrt{\nu n}}\right)\Bigg|.
\end{multline*}
Our strategy is to prove that each of these terms can be made arbitrarily small by taking $\tilde\De\to0$.

For $B_1$, if we let
\begin{align*}
I^1_{t,\tilde\De}(y,z)&:=\int_{(1-t)(1-\tilde\De)}^{(1-t)}\frac{y\ex^{-y^2/2u}}{\sqrt{2\pi u^3}}
\frac{z\ex^{-z^2/2(1-t-u)}}{\sqrt{2\pi (1-t-u)^3}}\d u\quad\text{and}\\
I^2_{t,\tilde\De}(y,z)&:=\int_0^{\tilde\De}\frac{y\ex^{-y^2/2u}}{\sqrt{2\pi u^3}}
\frac{z\ex^{-z^2/2(1-t-u)}}{\sqrt{2\pi (1-t-u)^3}}\d u,
\end{align*}
then it is enough to prove that for any fixed $\De$ and $\tilde\ka>0$, one has
\begin{align}\label{Equation: tilde A1 reduced}
\lim_{\tilde\De\to0}\sup_{\substack{y,z\geq\tilde\ka \\ 0\leq t\leq 1-\De}}
I^i_{t,\tilde\De}(y,z)=0,\qquad i=1,2.
\end{align}
Let us begin with the case $i=1$.
For every $0<u\leq 1$, the function
\[y\mapsto \frac{y\ex^{-y^2/2u}}{\sqrt{2\pi u^3}}\]
attains its maximum at $y=\sqrt{u}$, in which case the maximal value is $(2\ex\pi u^2)^{-1/2}$.
Therefore,
\[I^1_{t,\tilde\De}(y,z)\leq\frac{1}{\sqrt{2\ex\pi}\De(1-\tilde\De)}\int_{(1-t)(1-\tilde\De)}^{(1-t)}
\frac{z\ex^{-z^2/2(1-t-u)}}{\sqrt{2\pi (1-t-u)^3}}\d u,\]
since $u\geq(1-t)(1-\tilde\De)\geq\De(1-\tilde\De)$.
Moreover, for $z\geq\tilde\ka$ and $t\geq0$,
\begin{align}\label{Equation: tilde A1 reduced 2}
\int_{(1-t)(1-\tilde\De)}^{1-t}
\frac{z\ex^{-z^2/2(1-t-u)}}{\sqrt{2\pi (1-t-u)^3}}\d u=\erfc\left(\frac{z}{\sqrt{2(1-t)\tilde\De}}\right)\leq
\erfc\left(\frac{\tilde\ka}{\sqrt{2\tilde\De}}\right).
\end{align}
Combining these two bounds gives \eqref{Equation: tilde A1 reduced} for $i=1$.
By reversing the role of $y$ and $z$ in those upper bounds, we may similarly prove
\eqref{Equation: tilde A1 reduced} in the case $i=2$.

For $B_2$, if we combine the fact that $y\mapsto |y|\ex^{-y^2/2}/\sqrt{2\pi}$ is uniformly bounded in $y$
with $\ell\geq\ka\sqrt{n}$, $m\geq(n-k)(1-\tilde\De)$, and $k\leq(1-\De)n$,
then \eqref{Equation: Omega Local Limit} implies that
\begin{align}\label{Equation: Omega Bound for Large m}
\Pr[\Om_{\ell-1}=m]
\leq\frac1n\frac{C}{(m/n)}
\leq\frac1n\frac{C}{(1-k/n)(1-\tilde\De)}
\leq\frac1n\frac{C}{\De(1-\tilde\De)}
\end{align}
for $C>0$ independent of all parameters. Moreover,
\[\sum_{m=\lfloor(n-k)(1-\tilde\De)\rfloor+1}^{n-k}\Pr[\tau_{x-a}=n-k-m]\leq\Pr[\tau_{x-a}\leq n\tilde\De].\]
Since $|x-a|\geq\ka\sqrt{n}$, combining these bounds with Remark
\ref{Remark: Hitting Time Away from Zero Rescaled} implies that $B_2\to0$ as $\tilde\De\to0$.

For $B_3$, by combining \eqref{Equation: Hitting Away Local}, $m\leq\tilde\De n$, and $k\leq(1-\De)n$,
we obtain the bound
\begin{multline*}
n\Pr[\tau_{x-a}=n-k-m]=\frac{n}{n-k-m}(n-k-m)\Pr[\tau_{x-a}=n-k-m]\\
\leq\frac{C}{1-k/n-m/n}\leq\frac{C}{\De-\tilde\De}
\end{multline*}
for $C$ independent of all parameters.
Moreover, it follows from \eqref{Equation: Omega Local Limit} that
\[\sum_{k=1}^{\lfloor\tilde\De n\rfloor}\Pr[\Om_{\ell-1}=k]=\frac1n\sum_{k=1}^{\lfloor\tilde\De n\rfloor}\frac{(\nu\ell)\ex^{-(\nu\ell)^2/2\nu n(k/n)}}{\sqrt{2\pi\nu n(k/n)^3}}+o(1)\]
with $o(1)$ uniform in $\ell\geq\ka\sqrt{n}$.
Given that
\[f(u,y):=\frac{y\ex^{-y^2/2\nu u}}{\sqrt{2\pi\nu u^3}},\qquad (u,y)\in[0,1]\times[\ka,\infty)\]
is uniformly continuous on its specified domain, it follows from Lemma \ref{Lemma: Uniform Riemann Approximation} that
\[\sum_{k=1}^{\lfloor\tilde\De n\rfloor}\Pr[\Om_{\ell-1}=k]
=\int_0^{\tilde\De}\frac{(\nu\ell)\ex^{-(\nu\ell)^2/2\nu nu}}{\sqrt{2\pi\nu nu^3}}\d u+o(1)\]
with $o(1)$ uniform in $\ell\geq\ka\sqrt{n}$.
This can then be controlled as $\tilde\De\to0$
in the same way as the term $A_2$ in the proof of Lemma \ref{Lemma: Tail Bounds}.

Finally, for $B_4$, if we write
\begin{multline*}
n\Pr[\Om_{\ell-1}=m]\Pr[\tau_{x-a}=n-k-m]\\
=\frac{1}{n}n\Pr[\Om_{\ell-1}=m]\frac{n-k-m}{1-k/n-m/n}\Pr[\tau_{x-a}=n-k-m],
\end{multline*}
then we see that $B_4=0$ for any fixed $\tilde\De$ by combining
\eqref{Equation: Hitting Away Local}, \eqref{Equation: Omega Local Limit},
and Lemma \ref{Lemma: Uniform Riemann Approximation} applied to the function
\[f(u;y,z):=\frac{y\ex^{-y^2/2u}}{\sqrt{2\pi u^3}}
\frac{z\ex^{-z^2/2(1-t-u)}}{\sqrt{2\pi (1-t-u)^3}},\qquad (u;y,z)\in[\tilde\De,1-\tilde\De]\times\RR^2.\]

\subsubsection{Step 2}

With \eqref{Equation: Main 1 Proof Part 1} proved, if we define
\[I_\De(y,z):=\int_{\De}^{1-\De}\frac{y\ex^{-y^2/2u}}{\sqrt{2\pi u^3}}
\frac{z\ex^{-z^2/2(1-t-u)}}{\sqrt{2\pi (1-t-u)^3}}\d u,\qquad y,z>0,~\De\in[0,1/2)\]
so that $I_0(y,z)=(y+z)\ex^{-(y+z)^2/2}/\sqrt{2\pi}$ by Lemma \ref{Lemma: Integral 1},
then we need only show that the following can be made arbitrarily small by taking $\De\to0$.
\begin{align*}
C_1&:=\sup_{n\in\NN}\sup_{\substack{x,a\in\ZZ,~\ell\in\NN \\ |a|,|x-a|,\ell\geq\ka\sqrt{n}}}
\left|I_{\De}\left(\frac{|a|}{\sqrt{\nu n}},\frac{|x-a|+\nu\ell}{\sqrt{\nu n}}\right)
-I_{0}\left(\frac{|a|}{\sqrt{\nu n}},\frac{|x-a|+\nu\ell}{\sqrt{\nu n}}\right)\right|,\\
C_2&:=\limsup_{n\to\infty}\sup_{\substack{x,a\in\ZZ,~\ell\in\NN \\ |x-a|\geq\ka\sqrt{n}}}
n\sum_{k=\lfloor(1-\De)n\rfloor+1}^n\Pr[\tau_a=k]\\
&\hspace{2.2in}\cdot\sum_{m=1}^{n-k}\Pr[\Om_{\ell-1}=m]\Pr[\tau_{x-a}=n-k-m],\\
C_3&:=\lim_{n\to\infty}\sup_{\substack{x,a\in\ZZ,~\ell\in\NN \\ \ell\geq\ka\sqrt{n}}}
\Bigg|\sum_{k=\lfloor n\De\rfloor+1}^{\lfloor(1-\De)n\rfloor}
\Pr[\tau_a=k]\frac{(|x-a|+\nu\ell)\ex^{-(|x-a|+\nu\ell)^2/2\nu n(1-k/n)}}{\sqrt{2\pi\nu n(1-k/n)^3}}\\
&\hspace{2.55in}-I_{\De}\left(\frac{|a|}{\sqrt{\nu n}},\frac{|x-a|+\nu\ell}{\sqrt{\nu n}}\right)\Bigg|,\text{ and}\\
C_4&:=\limsup_{n\to\infty}\sup_{\substack{x,a\in\ZZ,~\ell\in\NN \\ |a|,|x-a|,\ell\geq\ka\sqrt{n}}}
\sum_{k=1}^{\lfloor n\De\rfloor}\Pr[\tau_a=k]
\frac{(|x-a|+\nu\ell)\ex^{-(|x-a|+\nu\ell)^2/2\nu n(1-k/n)}}{\sqrt{2\pi\nu n(1-k/n)^3}}.
\end{align*}

For $C_1$, we proceed as in \eqref{Equation: tilde A1 reduced}.

For $C_2$, $k\geq(1-\De)n$ implies by \eqref{Equation: Hitting Away Local} that $\Pr[\tau_a=k]\leq C/(1-\De)n$
for $C>0$ independent of all parameters. Thus, it is enough to control
\[\sup_{\lfloor(1-\De)n\rfloor+1\leq k\leq n}
\sum_{m=1}^{n-k}\Pr[\Om_{\ell-1}=m]\Pr[\tau_{x-a}=n-k-m].\]
With the trivial bound $\Pr[\Om_{\ell-1}=m]\leq 1$, this is bounded by
\[\sum_{m=1}^{n-\lfloor(1-\De)n\rfloor-1}\Pr[\tau_{x-a}=\De n-\lfloor(1-\De)n\rfloor-1-m]\leq\Pr[\tau_{x-a}\leq\De n].\]
Taking a supremum over $|x-a|\geq\ka\sqrt{n}$, we can prove that this vanishes as $\De\to0$ by Remark
\ref{Remark: Hitting Time Away from Zero Rescaled}.

For $C_3$, we need only combine \eqref{Equation: Hitting Away Local} with Lemma
\ref{Lemma: Uniform Riemann Approximation}.

Finally, for $C_4$, if we use the bound
\[\frac{(|x-a|+\nu\ell)\ex^{-(|x-a|+\nu\ell)^2/2\nu n(1-k/n)}}{\sqrt{2\pi\nu n(1-k/n)^3}}\leq\frac{C}{(1-k/n)}
\leq\frac{C}{1-\De},\]
valid for $k\leq\De n$ and any $x,a,\ell$, and $n$,
then we need only control
\[\sum_{k=1}^{\lfloor n\De\rfloor}\Pr[\tau_a=k]=\Pr[\tau_a\leq\De n].\]
Taking a supremum over $|a|\geq\ka\sqrt{n}$, this vanishes as $\De\to0$.

\subsection{Proof of \eqref{Equation: Main Theorem 2 in Section}}\label{Section: Main Theorem 2}

The technical details of the proof of \eqref{Equation: Main Theorem 2 in Section} are virtually the same as
in \eqref{Equation: Main Theorem 1 in Section}, the only difference being that the terms
$\Pr[\tau_{x-a}=n-k-m]$ are replaced by $\Pr[\tau_0>n-m-k]$. Consequently, we omit the proof.
To illustrate the need for a different normalization (i.e., $\sqrt{n}$ instead of $n$) and the appearance of
the function $\psi_\nu(a,\ell):=2\nu\ex^{-(|a|+\nu\ell)^2/2\nu}/\sqrt{2\pi\nu}$, we offer the following heuristic.
By combining \eqref{Equation: Hitting Away Local}, \eqref{Equation: Omega Local Limit}, and
\eqref{Equation: Hitting Zero Deviation} as well as Lemma \ref{Lemma: Integral 2}, we see that
\begin{align*}
&\sqrt{n}\sum_{k=1}^n\Pr[\tau_a=k]\sum_{m=1}^{n-k}\Pr[\Om_{\ell-1}=m]\Pr[\tau_0>n-k-m]\\
&=\frac1n\sum_{k=1}^n\frac{k}{k/n}\Pr[\tau_a=k]\frac1n\sum_{m=1}^{n-k}n\Pr[\Om_{\ell-1}=m]
\frac{\sqrt{n-k-m}}{\sqrt{1-k/n-m/n}}\Pr[\tau_0>n-k-m]\\
&\approx\int_0^1\frac{|a|\ex^{-a^2/2\nu nt}}{\sqrt{2\pi\nu n t^3}}
\int_0^{1-t}\frac{(\nu\ell)\ex^{-(\nu\ell)^2/2\nu nu}}{\sqrt{2\pi\nu n u^3}}\frac{2\sqrt\nu}{\sqrt{2\pi(1-t-u)}}\d u\dd t\\
&=2\nu\frac{\ex^{-(|a|+\nu\ell)^2/2\nu n}}{\sqrt{2\pi\nu}}.
\end{align*}

\section{Discussion}\label{Section: Discussion}

\subsection{Stable Laws and Infinite Variance}

The most fundamental ingredient in the proof of Theorem \ref{Theorem: Main} (i.e., Lemma \ref{Lemma: Path Transformation})
does not require that $X$ has finite variance. Thus, in principle, it seems that the methods in this note could be used
to prove analogous statements in the case where $X$ has a lattice distribution in the domain of attraction of some
general stable law (see, for instance, \cite{VaWa} for conditioned local limit theorems that hold in this level of
generality).

Much of the technical difficulties in treating this general case comes from the necessity of dealing with
normalizations of the form $n^{\al}v(n)$, where $v$ is a slowly varying function, as well as limiting densities
that have no closed-form formula. These problems are compounded by the fact that many of the results that are
essential in our arguments, such as \cite[Theorem 8]{Ke} and every result in \cite{Ka}, have not been proved in
such generality. Thus using the methods in this note to extend Theorem
\ref{Theorem: Main} to general stable laws appears to be a challenging undertaking.

\subsection{Nonlattice Random Walks}

Another interesting generalization of this note would be to extend Theorem \ref{Theorem: Main} for nonlattice random variables
and hopefully prove a statement in the style of \cite[Theorem 1]{Ca} or \cite[Theorem 3]{VaWa}.
Of course, if $X$ is nonlattice, then the definition of the occupation measure that we use must be different from
\eqref{Equation: Occupation Measure}.

The existing strong invariance principles involving nonlattice occupation measures
(such as \cite[Section 4]{BK}) suggest that we use an occupation measure of the form
\[\sum_{k=1}^n\1_{\{S_k\in[a-\eps,a+\eps)\}},\qquad a\in\RR,~n\in\NN,\]
where $\eps>0$ is a fixed constant (see Figure \ref{Figure: Nonlattice Local Time} below).
\begin{figure}[htbp]
\begin{center}
\begin{tikzpicture}

\begin{axis}[%
width=4in,
height=2in,
at={(1.011in,0.642in)},
scale only axis,
xmin=0,
xmax=100,
ymin=-8,
ymax=6,
axis background/.style={fill=white},
axis x line*=bottom,
axis y line*=left
]
\addplot [color=blue, forget plot,thick]
  table[row sep=crcr]{%
0	0\\
1	0.499932667897628\\
2	0.0794232343084948\\
3	1.15876960309019\\
4	1.27248077826823\\
5	0.755384296672616\\
6	2.27613031664442\\
7	3.57843441713669\\
8	3.75218108558052\\
9	4.17644722881151\\
10	4.47797893036798\\
11	3.46556853458801\\
12	2.77706562629463\\
13	2.67634115093162\\
14	1.74272475142173\\
15	2.93544539630396\\
16	1.87807787879546\\
17	0.928643077587864\\
18	-0.212057708139105\\
19	-1.15545625393246\\
20	-1.37820254618362\\
21	-2.03256342009376\\
22	-0.565933318067464\\
23	-0.807702006487329\\
24	-1.89953030101037\\
25	-0.496981483467815\\
26	1.16491564885522\\
27	0.953155024027938\\
28	-0.393967502147508\\
29	-1.23205597979736\\
30	-1.54825970830889\\
31	-1.21953006506542\\
32	-2.04325273363858\\
33	-1.68699382160435\\
34	-0.955320895461444\\
35	-1.91921848356968\\
36	-3.24452261716328\\
37	-3.94885805314414\\
38	-4.5766284301407\\
39	-4.83932248029723\\
40	-4.81210058371044\\
41	-6.2479159804599\\
42	-7.07070165486532\\
43	-6.02795641394859\\
44	-7.65878521081966\\
45	-6.17319089359525\\
46	-5.37530137956161\\
47	-5.4147610518067\\
48	-5.14274226108662\\
49	-6.0528186367233\\
50	-6.19537047749001\\
51	-4.59118472059443\\
52	-4.42904495471309\\
53	-4.35582828908767\\
54	-5.28561260760271\\
55	-5.32407195082053\\
56	-4.89431519900392\\
57	-4.27377148256231\\
58	-4.63571739272206\\
59	-5.09493031260721\\
60	-3.40451106729684\\
61	-5.00583060737182\\
62	-3.67156948805823\\
63	-2.23990192091799\\
64	-1.2138908860537\\
65	-2.60399232969787\\
66	-3.42889474626217\\
67	-3.99923538286843\\
68	-3.37663949621679\\
69	-4.63565636012081\\
70	-3.86930182497091\\
71	-5.23151869531103\\
72	-4.69888761544898\\
73	-4.71906969094678\\
74	-3.75240616579434\\
75	-3.00749587519199\\
76	-1.60896682930454\\
77	-0.254771550658556\\
78	-0.82924758752011\\
79	-0.140771828727251\\
80	-1.18758929618328\\
81	-2.81384316233043\\
82	-1.96834512277795\\
83	-1.9682674036137\\
84	-2.03781914689814\\
85	-0.635820188326823\\
86	-0.255230954076467\\
87	0.152377376144327\\
88	1.39752204768221\\
89	2.455768456603\\
90	2.72153958275964\\
91	1.62315079693596\\
92	0.722248854701422\\
93	2.06116546634036\\
94	0.428444836635099\\
95	0.393462220269646\\
96	-0.756871890715853\\
97	0.901326520841275\\
98	1.63812178371507\\
99	1.6397555377116\\
100	1.53960272926608\\
};
\addplot [color=red, dashed, forget plot]
  table[row sep=crcr]{%
0	0.5\\
100	0.5\\
};
\addplot [color=red, dashed, forget plot]
  table[row sep=crcr]{%
0	-0.5\\
100	-0.5\\
};
\addplot [color=black, draw=none, mark=x, mark options={solid, black}, forget plot, thick]
  table[row sep=crcr]{%
1	0.499932667897628\\
};
\addplot [color=black, draw=none, mark=x, mark options={solid, black}, forget plot, thick]
  table[row sep=crcr]{%
2	0.0794232343084948\\
};
\addplot [color=black, draw=none, mark=x, mark options={solid, black}, forget plot, thick]
  table[row sep=crcr]{%
18	-0.212057708139105\\
};
\addplot [color=black, draw=none, mark=x, mark options={solid, black}, forget plot, thick]
  table[row sep=crcr]{%
25	-0.496981483467815\\
};
\addplot [color=black, draw=none, mark=x, mark options={solid, black}, forget plot, thick]
  table[row sep=crcr]{%
28	-0.393967502147508\\
};
\addplot [color=black, draw=none, mark=x, mark options={solid, black}, forget plot, thick]
  table[row sep=crcr]{%
77	-0.254771550658556\\
};
\addplot [color=black, draw=none, mark=x, mark options={solid, black}, forget plot, thick]
  table[row sep=crcr]{%
79	-0.140771828727251\\
};
\addplot [color=black, draw=none, mark=x, mark options={solid, black}, forget plot, thick]
  table[row sep=crcr]{%
86	-0.255230954076467\\
};
\addplot [color=black, draw=none, mark=x, mark options={solid, black}, forget plot, thick]
  table[row sep=crcr]{%
87	0.152377376144327\\
};
\addplot [color=black, draw=none, mark=x, mark options={solid, black}, forget plot, thick]
  table[row sep=crcr]{%
94	0.428444836635099\\
};
\addplot [color=black, draw=none, mark=x, mark options={solid, black}, forget plot, thick]
  table[row sep=crcr]{%
95	0.393462220269646\\
};
\end{axis}
\end{tikzpicture}%
\caption{Nonlattice Occupation Measure of Level $a=0$.}
\label{Figure: Nonlattice Local Time}
\end{center}
\end{figure}
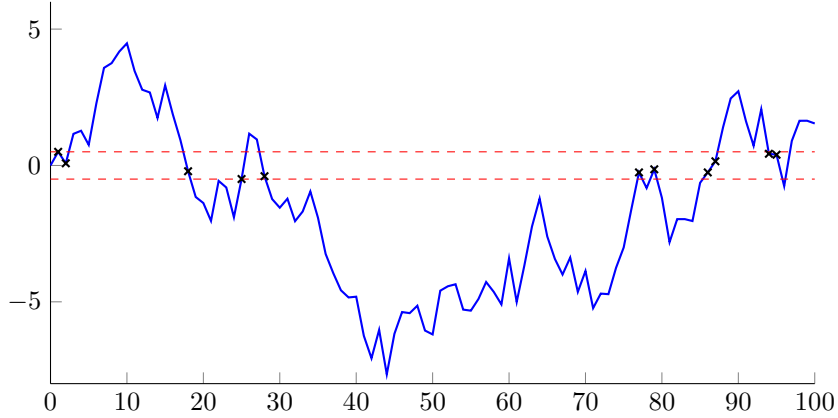
However, it is clear that a path decomposition similar to the one described in Figure \ref{Figure: Path Decomposition} is not as well behaved in this case. For example, given that successive visits of the walk to $[a-\eps,a+\eps)$ occur at different levels, the time increments between these visits are not i.i.d.

Consequently, it appears that a different path transformation or decomposition, a different occupation measure,
or an altogether different approach may be required to tackle this problem.

\appendix

\section{Integral Identities}

\begin{lemma}\label{Lemma: Integral 1}
For every $y,z,t>0$,
\[\int_0^t\frac{y\ex^{-y^2/2u}}{\sqrt{2\pi u^3}}\frac{z\ex^{-z^2/2(t-u)}}{\sqrt{2\pi(t-u)^3}}\d u
=\frac{(y+z)\ex^{-(y+z)^2/2t}}{\sqrt{2\pi t^3}}.\]
\end{lemma}
\begin{proof}
If $T_y$ denotes the first hitting time of $y>0$ by a standard Brownian motion, then it is well known that
\[\Pr[T_y\in\dd t]=\frac{y\ex^{-y^2/2t}}{\sqrt{2\pi t^3}}\d t.\]
According to the strong Markov property, if $T_y$ and $T_z$ are independent Brownian
hitting times, then $T_y+T_z$ is equal in distribution to $T_{y+z}$. Therefore, the result follows
directly by a convolution computation.
\end{proof}

\begin{lemma}\label{Lemma: Integral 2}
For every $y,t>0$ and $z\geq0$,
\[\int_0^t\frac{y\ex^{-y^2/2u}}{\sqrt{2\pi u^3}}\frac{\ex^{-z^2/2(t-u)}}{\sqrt{2\pi(t-u)}}\d u
=\frac{\ex^{-(y+z)^2/2t}}{\sqrt{2\pi t}}.\]
\end{lemma}
\begin{proof}
Note that
\[\frac{\partial}{\partial z}\int_0^t\frac{y\ex^{-y^2/2u}}{\sqrt{2\pi u^3}}\frac{\ex^{-z^2/2(t-u)}}{\sqrt{2\pi(t-u)}}\d u
=-\int_0^t\frac{y\ex^{-y^2/2u}}{\sqrt{2\pi u^3}}\frac{z\ex^{-z^2/2(t-u)}}{\sqrt{2\pi(t-u)^3}}\d u.\]
Therefore, Lemma \ref{Lemma: Integral 1} implies that
\begin{align}\label{Equation: Integral 2}
\int_0^t\frac{y\ex^{-y^2/2u}}{\sqrt{2\pi u^3}}\frac{\ex^{-z^2/2(t-u)}}{\sqrt{2\pi(t-u)}}\d u
=\frac{\ex^{-(y+z)^2/2t}}{\sqrt{2\pi t}}+\Phi(y),\qquad y,z,t>0
\end{align}
for some function $\Phi$.

We note that for any $y,t>0$,
\[\frac{\dd}{\dd u}
\left(-\frac{\ex^{-y^2/2t}}{\sqrt{2 \pi  t}}\erf\left(\frac{y}{\sqrt{2 t}}\frac{\sqrt{t-u}}{\sqrt u}\right)\right)
=\frac{y\ex^{-y^2/2u}}{\sqrt{2\pi u^3}}\frac{1}{\sqrt{2\pi(t-u)}}.\]
Moreover,
\[\lim_{u\to t}\frac{\ex^{-y^2/2t}}{\sqrt{2 \pi  t}}\erf\left(\frac{y}{\sqrt{2 t}}\frac{\sqrt{t-u}}{\sqrt u}\right)=0\]
and
\[\lim_{u\to0}\frac{\ex^{-y^2/2t}}{\sqrt{2 \pi  t}}\erf\left(\frac{y}{\sqrt{2 t}}\frac{\sqrt{t-u}}{\sqrt u}\right)
=\frac{\ex^{-y^2/2t}}{\sqrt{2\pi t}}.\]
This proves the lemma for $z=0$ by the fundamental theorem of calculus. If we then extend
\eqref{Equation: Integral 2} to $z=0$ by continuity, this also implies that $\Phi(y)=0$ for all $y$, concluding the
proof of the lemma for $z>0$.
\end{proof}

\bibliographystyle{plain}
\bibliography{Bibliography} 
\end{document}